\documentclass[11pt]{amsart}
\usepackage{amsmath,amssymb,amsfonts,amsthm,amscd,indentfirst}
\usepackage{hyperref}\usepackage{xcolor}
\usepackage{setspace}

\numberwithin{equation}{section}

\newtheorem{prop}{Proposition}[section]
\newtheorem{theo}[prop]{Theorem}
\newtheorem{lemm}[prop]{Lemma}
\newtheorem{coro}[prop]{Corollary}
\newtheorem{rema}[prop]{Remark}

\newtheorem{ques}[prop]{Question}
\newtheorem*{thea}{Theorem A}
\newtheorem*{theob}{Theorem B}
\newtheorem*{theoc}{Theorem C}

\def\and{\quad{\rm and}\quad}

\def\<{\langle}
\def\>{\rangle}

\title[Existence and nonexistence to exterior Dirichlet problem]{Existence and nonexistence to exterior Dirichlet problem For Monge-Amp\`ere equation}
\author{ Yanyan Li
        \and
        Siyuan Lu 
        }       
\address{Department of Mathematics, Rutgers University, 110 Frelinghuysen Road, Piscataway, NJ 08854}
\email{yyli@math.rutgers.edu}
\email{siyuan.lu@math.rutgers.edu}
\thanks{Research of the first named author is partially supported by NSF grant DMS-1501004. Research of the second named author is partially supported by CSC fellowship.}
\usepackage{graphicx}

\begin{document}

\begin{abstract}
We consider the exterior Dirichlet problem for Monge-Amp\`ere equation with prescribed asymptotic behavior. Based on earlier work by Caffarelli and the first named author, we complete the characterization of the existence and nonexistence of solutions in terms of their asymptotic behaviors.
\end{abstract}


\maketitle

\section{Introduction}

A classic theorem of J\"orgens \cite{J}, Calabi \cite{Calabi} and Pogorelov \cite{P} states that any classical convex solution of
\begin{align*}
\det(D^2u)=1\quad in \quad \mathbb{R}^n
\end{align*}
must be a quadratic polynomial.

A simpler and more analytical proof, along the lines of affine geometry, was later given by Cheng and Yau \cite{CY}. The theorem was extended by Caffarelli \cite{C1} to viscosity solutions. Another proof of the theorem was given by Jost and Xin \cite{JX}. Trudinger and Wang \cite{TW} proved that if $\Omega$ is an open convex subset of $\mathbb{R}^n$ and $u$ is a convex $C^2$ solution of $\det(D^2u)=1$ in $\Omega$ with $\lim_{x\rightarrow \partial \Omega}u(x)=\infty$, then $\Omega=\mathbb{R}^n$.

In 2003, Caffarelli and the first named author \cite{CL} extended the J\"orgens-Calabi-Pogorelov theorem to exterior domain. In a subsequent paper \cite{CL2}, they also gave another extension: a classical convex solution $u$ of $\det(D^2u)=f$ in $\mathbb{R}^n$ with a periodic positive $f$ must be sum of a quadratic polynomial and a periodic function. Moreover, their approach in \cite{CL} is enough to establish the existence of solutions to the exterior Dirichlet problem for Monge-Amp\`ere equations. More specifically, let
\begin{align*}
\mathcal{A}=\{ A| A \text{ is a real }n \times n \text{ symmetric positive definite matrix with} \det A=1\}.
\end{align*}

\begin{thea}[\cite{CL}]\label{thm cl}
Let $D$ be a smooth, bounded, strictly convex domain in $\mathbb{R}^n$ for $n\geq 3$ and $\varphi\in C^2(\partial D)$. Then for any $A\in \mathcal{A}$, $b\in \mathbb{R}^n$, there exists a constant $c_1$ depending only on $n,D,\varphi,b$ and $A$, such that for every $c>c_1$, there exists a unique solution $u\in C^\infty(\mathbb{R}^n\setminus \bar{D})\cap C^0(\overline{\mathbb{R}^n\setminus D})$ satisfying
\begin{align*}
\begin{cases}
\det(D^2u)=1 \quad in \quad \mathbb{R}^n\setminus \bar{D},\\
u=\varphi\quad on \quad \partial D,\\
\lim_{|x|\rightarrow\infty}|u(x)-\left(\frac{1}{2}x^\prime Ax+bx+c\right)|=0.
\end{cases}
\end{align*}
\end{thea}

For $n=2$, the exterior Dirichlet problem was studied earlier by Ferrer, Mart\'inez and Mil\'an \cite{FMM1,FMM2} using complex variable methods.

\medskip

It is desirable to completely understand the relation between the constant $c$ and the exsitence and nonexistence of solutions to the exterior Dirichlet problem. More precisely,

\begin{ques}\label{q}
Concerning Theorem A, is there a sharp constant $C_*$ such that we have the existence of solutions for $c\geq C_*$ while have the nonexistence for $c<C_*$?
\end{ques}

\medskip

In \cite{WB}, Wang and Bao answered the question among radially symmetric solutions. A special case of their theorem is as follows:
\begin{theob}[\cite{WB}]\label{them-wb}
For $n\geq 3$, let $C_*=-\frac{1}{2}+\int_1^\infty s((1-s^{-n})^{\frac{1}{n}}-1)ds$. Then for every $c\geq C_*$, there exists a unique radially symmetric solution $u\in C^\infty(\mathbb{R}^n\setminus \overline{B_1(0)})\cap C^1(\overline{\mathbb{R}^n\setminus B_1(0)})$ satisfying
\begin{align*}
\begin{cases}
\det(D^2u)=1 \quad in \quad \mathbb{R}^n\setminus \overline{B_1(0)}, \\
u=0\quad on \quad \partial B_1(0),\\
\lim_{|x|\rightarrow\infty}|u(x)-\left(\frac{1}{2}|x|^2+c\right)|=0.
\end{cases}
\end{align*}
While for $c<C_*$, there is no radially symmetric classical solution for the above problem. 
\end{theob}

\medskip

In this paper, we give an affirmative answer to Question \ref{q}. To begin with, let us recall the definition of viscosity solutions.

Let $\Omega$ be a domain in $\mathbb{R}^n$, $g\in C^0(\Omega) $ be a positive function, and $u\in C^0(\Omega)$ be a locally convex function. We say that $u$ is a viscosity subsolution of
\begin{align}\label{vis}
\det(D^2u)=g\quad in \quad \Omega
\end{align}
if for any $\bar{x}\in \Omega$ and every convex function $\varphi\in C^2(\Omega)$ satisfying
\begin{align*}
\varphi\geq u\quad in\quad \Omega\quad and \quad\varphi(\bar{x})=u(\bar{x}),
\end{align*}
we have
\begin{align*}
\det(D^2\varphi(\bar{x}))\geq g(\bar{x}).
\end{align*}

Similarly, we say $u$ is a viscosity supsolution of (\ref{vis}) if for any $\bar{x}\in \Omega$ and every convex function $\varphi\in C^2(\Omega)$ satisfying
\begin{align*}
\varphi\geq u\quad in\quad \Omega\quad and \quad\varphi(\bar{x})=u(\bar{x}),
\end{align*}
we have
\begin{align*}
\det(D^2\varphi(\bar{x}))\leq g(\bar{x}).
\end{align*}

We say $u$ is a viscosity solution of (\ref{vis}) if $u$ is both a viscosity subsolution and a viscosity supsolution of (\ref{vis}).

\medskip

\begin{theo}\label{Main thm}
Let $D$ be a smooth, bounded, strictly convex domain in $\mathbb{R}^n$ for $n\geq 3$ and $\varphi\in C^2(\partial D)$. Then for any $A\in \mathcal{A}$, $b\in \mathbb{R}^n$, there exists a constant $C_*$ depending only on $n,D,\varphi,b$ and $A$, such that for every $c\geq C_*$, there exists a unique solution $u\in C^\infty(\mathbb{R}^n\setminus \bar{D})\cap C^0(\overline{\mathbb{R}^n\setminus D})$ satisfying
\begin{align*}
\begin{cases}
\det(D^2u)=1 \quad in \quad \mathbb{R}^n\setminus \bar{D},\\
u=\varphi\quad on \quad \partial D,\\
\lim_{|x|\rightarrow\infty}|u(x)-\left(\frac{1}{2}x^\prime Ax+bx+c\right)|=0.
\end{cases}
\end{align*}
While for $c<C_*$, there is no viscosity solution for the above problem.
\end{theo}

\medskip

Apart from the exterior Dirichlet problem with constant $1$ on the right hand side, it is natural to consider the problem with a general right hand side. In the case that the right hand side is an appropriate perturbation of $1$ near infinity, the existence and uniqueness ware given by Bao, Li and Zhang \cite{BLZ}. More precisely, assume that
\begin{align*}
g\in  C^0(\mathbb{R}^n),\quad 0<\inf_{\mathbb{R}^n} g\leq \sup_{\mathbb{R}^n} g<\infty,
\end{align*}
and for some integer $m\geq 3$ and some constant $\beta>2$ such that $D^mg$ exists outside a compact subset of $\mathbb{R}^n$ and
\begin{align*}
\lim_{|x|\rightarrow\infty}|x|^{\beta+|\alpha|}|D^\alpha(g(x)-1)|<\infty,\quad |\alpha|=0,\cdots,m. 
\end{align*}

\begin{theoc}[\cite{BLZ}]\label{thm blz}
Let $D$ be a smooth, bounded, strictly convex domain in $\mathbb{R}^n$ for $n\geq 3$, $\varphi\in C^2(\partial D)$ and $g$ satisfy the above assumption. Then for any $A\in \mathcal{A}$ , $b\in \mathbb{R}^n$, there exists a constant $c_1$ depending only on $n,D,\varphi,b, g$ and $A$, such that for every $c>c_1$, there exists a unique viscosity solution $u$ satisfying
\begin{align*}
\begin{cases}
\det(D^2u)=g \quad in \quad \mathbb{R}^n\setminus \bar{D},\\
u=\varphi\quad on \quad \partial D,\\
\lim_{|x|\rightarrow\infty}|u(x)-\left(\frac{1}{2}x^\prime Ax+bx+c\right)|=0.
\end{cases}
\end{align*}
\end{theoc}

Similar to Theorem \ref{Main thm}, we also give for this problem the characterization of the existence and nonexistence of solutions in terms of their asymptotic behavior for this problem.

\begin{theo}\label{thm perturbation}
Let $D$ be a smooth, bounded, strictly convex domain in $\mathbb{R}^n$ for $n\geq 3$, $\varphi\in C^2(\partial D)$ and $g$ satisfy the above assumption. Then for any $A\in \mathcal{A}$, $b\in \mathbb{R}^n$, there exists a constant $C_*$ depending only on $n,D,\varphi,b, g$ and $A$, such that for every $c\geq C_*$, there exists a unique viscosity solution $u$ satisfying
\begin{align*}
\begin{cases}
\det(D^2u)=g \quad in \quad \mathbb{R}^n\setminus \bar{D},\\
u=\varphi\quad on \quad \partial D,\\
\lim_{|x|\rightarrow\infty}|u(x)-\left(\frac{1}{2}x^\prime Ax+bx+c\right)|=0.
\end{cases}
\end{align*}
While for $c<C_*$, there is no viscosity solution for the above problem.
\end{theo}

\medskip

The main contribution of our results is that we give a complete characterization of the existence and nonexistecne of solutions to the exterior Dirichlet problem in terms of their asymptotic behavior. Note that this is different from the interior Dirichlet problem, for which the existence and uniqueness were established by the seminal work of Caffarelli, Nirenberg and Spruck \cite{CNS}. The special feather of the exterior Dirichlet problem is that it depends on the asymptotic behavior in a non-trivial way. Our result reveals the intimate relation between the asymptotic behavior and the existence and nonexistence. Moreover, our result is also sharp in the sense that all viscosity solutions to the exterior Dirichlet problem is asymptotic to a quadratic polynomial by the result of Caffarelli and Li \cite{CL}.

\medskip

The organization of the paper is as follows: In section 2, we prove Theorem \ref{Main thm} and Theorem \ref{thm perturbation}. Theorems \ref{Main thm} is proved as follows. We first prove in Lemma \ref{lemm1} that for any given data $D,\varphi,A$ and $b$, there exists $c_2$ such that for any $c<c_2$ there exists no subsolution $u$ of
\begin{align}\label{intor}
\begin{cases}
\det(D^2u)\geq 1 \quad in \quad \mathbb{R}^n\setminus \bar{D},\\
u=\varphi \quad on \quad \partial D,
\end{cases}
\end{align}
having the asymptotic behavior $\frac{1}{2}x^\prime Ax+bx+c$.

Then we prove in Lemma \ref{lemm2} that if there is a solution of (\ref{intor}) with asymptotic behavior $\frac{1}{2}x^\prime Ax+bx+c_3$, then for every $c>c_3$, there is a solution of (\ref{intor}) with asymptotic behavior $\frac{1}{2}x^\prime Ax+bx+c$. Thus, in view of Theorem A, for any given data $D,\varphi,A$ and $b$, there exists a $C_*$ such that for every $c>C_*$ there exists a solution of (\ref{intor}) with asymptotic behavior $\frac{1}{2}x^\prime Ax+bx+c$, while for every $c<C_*$ there is no such solution. Finally we prove that there is such a solution for $c=C_*$. Theorem \ref{thm perturbation} is proved similarly.

In section 3, we extend some results to more general fully nonlinear elliptic equations.

\section{Monge-Amp\`ere equation}
In this section, we prove Theorem \ref{Main thm} and Theorem \ref{thm perturbation}. To begin with, we prove that there is no viscosity subsolution for $c$ very negative.

\begin{lemm}\label{lemm1}
Let $D$ be a smooth, bounded, strictly convex domain in $\mathbb{R}^n$ for $n\geq 3$ and $\varphi\in C^2(\partial D)$. Then for any $A\in \mathcal{A}$, $b\in \mathbb{R}^n$, there exists a constant $c_2$ depending only on $n,D,\varphi,b$ and $A$, such that for $c< c_2$, there is no viscosity subsolution $u$ satisfying
\begin{align*}
\begin{cases}
\det(D^2u)\geq 1 \quad in \quad \mathbb{R}^n\setminus \bar{D},\\
u=\varphi \quad on \quad \partial D,\\
\lim_{|x|\rightarrow\infty}|u(x)-\left(\frac{1}{2}x^\prime Ax+bx+c\right)|=0.
\end{cases}
\end{align*}
\end{lemm}

\begin{proof}
Note that by an affine transformation and adding a linear function to $u$, we only need to consider $A=I$ and $b=0$. Thus without loss of generality, we may assume that $u$ has the asymptotic $\frac{|x|^2}{2}+c$,  $\varphi\geq 0$ and $D\subset B_1(0)$ such that there exists $x_0\in \partial D\cap \partial B_1(0)$.

Consider 
\begin{align}\label{defv}
v(x)=\sup\{u(Tx)| T:\mathbb{R}^n\rightarrow \mathbb{R}^n \text{ is an orthogonal transformation} \}.
\end{align}

For any compact set $K\subset \mathbb{R}^n\setminus \bar{B}_1$, we have 
\begin{align*}
|u|+|\nabla u|\leq C(K)\quad \text{in} \quad  K.
\end{align*}

It is clear from the definition of $v$ that
\begin{align*}
|v|+|\nabla v|\leq C(K) \quad \text{in} \quad K.
\end{align*}

By a standard argument, $v$ is a locally convex viscosity subsolution satisfying
\begin{align*}
\begin{cases}
\det(D^2v)\geq 1 \quad in \quad \mathbb{R}^n\setminus \overline{B_1(0)}, \\
v\geq 0 \quad on \quad \partial B_1(0),\\
\lim_{|x|\rightarrow\infty}|v(x)-\left(\frac{1}{2}|x|^2+c\right)|=0.
\end{cases}
\end{align*}

The second line is due to the fact that $\varphi\geq 0$.

Clearly $v$ is radially symmetric.

We first prove that $v$ is monotonically nondecreasing in $[1,\infty)$. Suppose not, then since $v(r)\rightarrow \infty$ as $r\rightarrow \infty$, there exists $1<r_1<r_2<r_3<\infty$ such that $v(r_1)=v(r_3)>v(r_2)$, violating the local convexity of $v$.

We already know that $v$ is locally convex, so $v$ is locally Lipschitz. Thus $v^\prime$ exists almost everywhere and $v^\prime$ is monotonically nondecreasing. Since $v$ is monotonically nondecreasing, $v^\prime\geq 0$ almost everywhere. 

On the other hand, we know that monotone functions are differentiable almost everywhere, so $v^{\prime\prime}$ exists and $v^{\prime\prime}\geq 0$ almost everythere.

To proceed, we first prove that $v^{\prime\prime}\in L^1_{loc}[1,\infty)$. Indeed, let $g_m(s)=m(v^\prime(s+\frac{1}{m})-v^\prime(s))$, then by Fatou's lemma, we have
\begin{align*}
\int_t^r v^{\prime\prime}ds&\leq \liminf_{m\rightarrow \infty} \int_t^r g_m(s)ds=\liminf_{m\rightarrow \infty}\left(m\int_r^{r+\frac{1}{m}}v^\prime(s)ds-m\int_t^{t+\frac{1}{m}}v^\prime(s)ds\right)\\
&\leq 2\sup _{[1,r+1]}v^\prime <\infty
\end{align*} 

Secondly, we show that for all $r\geq 1$ such that $v^\prime(r)$ and $v^{\prime\prime}(r)$ are defined, we have
\begin{align*}
\lim_{h\rightarrow 0}\frac{v(r+h)-v(r)-v^\prime(r)h}{h^2}=\frac{1}{2} v^{\prime\prime}(r)
\end{align*}

Indeed, we have
\begin{align*}
v(r+h)&=v(r)+\int_0^{h}v^\prime(r+s)ds=v(r)+\int_0^{h}(v^\prime(r)+v^{\prime\prime}(r)s+o(s))ds\\
&=v(r)+v^\prime(r)h+\frac{1}{2} v^{\prime\prime}(r)h^2+o(h^2).
\end{align*}
By taking limit on $h$, we have the above formula.

For any $\epsilon>0$, we can choose $h_0$ small enough such that, 
\begin{align*}
v(r+h)\leq v(r)+v^\prime(r)h+\frac{1}{2}\left(v^{\prime\prime}(r)+\epsilon\right)h^2,
\end{align*}
for any $|h|\leq h_0$. Define
\begin{align*}
\phi(r+h)=v(r)+v^\prime(r)h+\frac{1}{2}\left(v^{\prime\prime}(r)+\epsilon\right)h^2.
\end{align*} 
Thus $\phi\in C^2(B_{h_0}(r))$ satisfies $\phi\geq v$ in $B_{h_0}(r)$, $\phi(r)=v(r)$, and therefore by the definition of viscosity subsolution,
\begin{align*}
\det(D^2\phi(r))\geq 1.
\end{align*}

Let $\epsilon\rightarrow 0$, we obtain
\begin{align*}
\det(D^2v)\geq 1\quad a.e.\quad  in \quad \mathbb{R}^n\setminus \bar{B}_1(0).
\end{align*}

Note that $v$ is radially symmetric, we have
\begin{align*}
(D^2v)=diag(v^{\prime\prime}, \frac{v^\prime}{r},\cdots,\frac{v^\prime}{r}) \quad a.e.\quad  in \quad \mathbb{R}^n\setminus \bar{B}_1(0).
\end{align*}

Thus we can write the above equation as
\begin{align*}
v^{\prime\prime}\frac{{v^\prime}^{n-1}}{r^{n-1}}\geq 1\quad a.e.\quad  in \quad [1,\infty).
\end{align*}

Integrating the above, for all $1\leq t<r<\infty$, we have
\begin{align}\label{s-n-1}
\int _t^r v^{\prime\prime}{v^\prime}^{n-1}ds\geq \int_t^r s^{n-1}ds.
\end{align}

Let $\varphi_\epsilon$ be the standard mollifier, then
\begin{align*}
v_\epsilon=v*\varphi_\epsilon
\end{align*}
is a smooth sequence converging to $v$ in $C^0_{loc}(1,\infty)$.

Since $v^{\prime\prime}\in L^1_{loc}[1,\infty)$, we have
\begin{align*}
v_\epsilon^\prime=v^\prime*\varphi_\epsilon,\quad v_\epsilon^{\prime\prime}=v^{\prime\prime}*\varphi_\epsilon,
\end{align*}
and $v^\prime_\epsilon\rightarrow v^\prime$, $ v^{\prime\prime}_\epsilon\rightarrow v^{\prime\prime}$,  a.e.  in $ (1,\infty)$.

By Fatou's lemma, for a.e. $1<t<r<\infty$, we have
\begin{align*}
\int_t^r v^{\prime\prime}{v^\prime}^{n-1}ds\leq \liminf_{\epsilon\rightarrow 0}\int_t^r v_\epsilon^{\prime\prime}{v_\epsilon^\prime}^{n-1}ds.
\end{align*}
i.e.
\begin{align*}
n\int_t^r v^{\prime\prime}{v^\prime}^{n-1}ds \leq \liminf_{\epsilon\rightarrow 0}\left({v_\epsilon^\prime}^n(r)-{v_\epsilon^\prime}^n(t)\right)={v^\prime}^n(r)-{v^\prime}^n(t).
\end{align*}

Together with (\ref{s-n-1}), for a.e. $1<t<r<\infty$, we have
\begin{align*}
{v^\prime}^n(r)-{v^\prime}^n(t)\geq r^n-t^n.
\end{align*}

Since $v^\prime(t)\geq 0$ and let $t\rightarrow 1^+$, we have, for almost all $r$,
\begin{align*}
v^\prime(r)\geq \left(r^n-1\right)^{\frac{1}{n}}.
\end{align*}

Now
\begin{align*}
v(r)=v(1)+\int_1^r v^\prime ds\geq \int_1^r\left(s^n-1\right)^{\frac{1}{n}}ds\geq \frac{1}{2}r^2+C,
\end{align*}
for $r$ large enough, here $C$ is a constant under control, see in \cite{CL}. It follows that that $c\geq C$.
\end{proof}

\medskip

We now prove the nonexistence in the case that $g$ is a perturbation of $1$ at infinity.
\begin{lemm}\label{lemm1-1}
Let $D$ be a smooth, bounded, strictly convex domain in $\mathbb{R}^n$ for $n\geq 3$,  $\varphi\in C^2(\partial D)$ and $g$ satisfy the assumption in Theorem \ref{thm perturbation}. Then for any $A\in \mathcal{A}$, $b\in \mathbb{R}^n$, there exists a constant $c_2$ dpending only on $n,D,\varphi,b,g$ and $A$, such that for $c< c_2$, there is no viscosity subsolution $u$ satisfying
\begin{align*}
\begin{cases}
\det(D^2u)\geq g \quad in \quad \mathbb{R}^n\setminus \bar{D},\\
u=\varphi \quad on \quad \partial D,\\
\lim_{|x|\rightarrow\infty}|u(x)-\left(\frac{1}{2}x^\prime Ax+bx+c\right)|=0.
\end{cases}
\end{align*}
\end{lemm}

\begin{proof}
Note that we can run the same argument as in Lemma \ref{lemm1}. Let 
\begin{align*}
g_1(r)=\sup_{|x|=r}g(x).
\end{align*}
Then for a.e. $1<t<r<\infty$, 
\begin{align*}
{v^\prime}^n(r)-{v^\prime}^n(t)\geq \int_t^r ns^{n-1}g_1(s)ds.
\end{align*}
Let $t\rightarrow 1^+$, we have for a.e. $r>1$, 
\begin{align*}
v^\prime(r)\geq \left(\int_1^r ns^{n-1}g_1(s)ds\right)^{\frac{1}{n}}.
\end{align*}
Thus for all $r>1$, 
\begin{align*}
v(r)\geq v(1)+\int_1^r\left(\int_1^l ns^{n-1}g_1(s)ds\right)^{\frac{1}{n}}dl.
\end{align*}
Now by Lemma 3.1 in \cite{BLZ}, we have
\begin{align*}
\int_1^r\left(\int_1^l ns^{n-1}g_1(s)ds\right)^{\frac{1}{n}}dl\geq \frac{1}{2}r^2+C,
\end{align*}
where $C$ is under control. It follows that that $c\geq C$.
\end{proof}

\medskip

As mentioned in the introduction, we have
\begin{coro}
For $n\geq 3$, let $u$ be a viscosity solution satisfying
\begin{align*}
\begin{cases}
\det(D^2u)=1 \quad in \quad \mathbb{R}^n\setminus \overline{B_1(0)}, \\
u=0\quad on \quad \partial B_1(0),\\
\lim_{|x|\rightarrow\infty}|u(x)-\left(\frac{1}{2}|x|^2+c\right)|=0.
\end{cases}
\end{align*}
Then $u$ is radially symmetric.
\end{coro}

\begin{proof}
Let $u$ be a viscosity solution for the above problem and let $v$ be defined as in (\ref{defv}). Then $v$ is a locally convex viscosity subsolution to the exterior Dirichlet problem
\begin{align*}
\begin{cases}
\det(D^2v)\geq 1 \quad in \quad \mathbb{R}^n\setminus \overline{B_1(0)}, \\
v= 0 \quad on \quad \partial B_1(0),\\
\lim_{|x|\rightarrow\infty}|v(x)-\left(\frac{1}{2}|x|^2+c\right)|=0.
\end{cases}
\end{align*}

Clearly $v$ is radially symmetric.

$\forall \epsilon>0$, there exists $R>1$ such that
\begin{align*}
u(x)+\epsilon\geq v(x), \quad |x|\geq R.
\end{align*}

Applying the comparison principle, see e.g. Proposition 2.1 in \cite{CL}, we have $u(x)+\epsilon\geq v(x)$ for $1\leq |x|\leq R$. Thus $u+\epsilon\geq v$ in $\mathbb{R}^n\setminus B_1(0)$. Sending $\epsilon$ to $0$, it gives $u\geq v$ in  $\mathbb{R}^n\setminus B_1(0)$.

On the other hand, by the definition of $v$, we have $v\geq u$. Thus $u=v$ is radially symmetric.

\end{proof}

\medskip

We now prove that if we have a viscosity solution with $c_3$, then we have a viscosity solution with all $c\geq c_3$.

\begin{lemm}\label{lemm2}
Let $D$ be a smooth, bounded, strictly convex domain in $\mathbb{R}^n$ for $n\geq 3$, $\varphi\in C^2(\partial D)$, $g$ satisfy the assumption in Theorem \ref{thm perturbation},  $A\in \mathcal{A}$ and $b\in \mathbb{R}^n$. Suppose that there exists a viscosity solution $u_3$ satisfying
\begin{align*}
\begin{cases}
\det(D^2u)=g \quad in \quad \mathbb{R}^n\setminus \bar{D},\\
u=\varphi \quad on \quad \partial D,\\
\lim_{|x|\rightarrow\infty}|u(x)-\left(\frac{1}{2}x^\prime Ax+bx+c_3\right)|=0.
\end{cases}
\end{align*}

Then for all $c\geq c_3$, there exists a viscosity solution $u$  satisfying
\begin{align*}
\begin{cases}
\det(D^2u)=g \quad in \quad \mathbb{R}^n\setminus \bar{D},\\
u=\varphi \quad on \quad \partial D,\\
\lim_{|x|\rightarrow\infty}|u(x)-\left(\frac{1}{2}x^\prime Ax+bx+c\right)|=0.
\end{cases}
\end{align*}
\end{lemm}

\begin{proof}
By Theorem C, there exists a viscosity solution $u$  satisfying
\begin{align*}
\begin{cases}
\det(D^2u)=g \quad in \quad \mathbb{R}^n\setminus \bar{D},\\
u=\varphi \quad on \quad \partial D,\\
\lim_{|x|\rightarrow\infty}|u(x)-\left(\frac{1}{2}x^\prime Ax+bx+c\right)|=0.
\end{cases}
\end{align*}
for all  $c\geq  c_1>c_3$. 

Thus we only need to prove that there exists a viscosity solution $u$ for all $c_3<c< c_1$. 

As before, we may assume without loss of generality that $A=I$ and $b=0$. 

We consider viscosity subsolutions $v$ to the exterior Dirichlet problem
\begin{align}\label{visdirichlet}
\begin{cases}
\det(D^2v)\geq g\quad in \quad \mathbb{R}^n\setminus \bar{D},\\
v= \varphi \quad on \quad \partial D,\\
\lim_{|x|\rightarrow\infty}|v(x)-\left(\frac{1}{2}|x|^2+c\right)|=0.
\end{cases}
\end{align}

First of all, we show that there exists at least one viscosity subsolution.

Let $u_3$ and $u_1$ be solutions with asymptotic behavior $\frac{1}{2}|x|^2+c_3$ and $\frac{1}{2}|x|^2+c_1$ respectively.

Note that $c=\alpha c_3+(1-\alpha) c_1$ for some $0<\alpha<1$ and let 
\begin{align*}
u_c:=\alpha u_3+(1-\alpha) u_1.
\end{align*}

By Alexandrov's theorem \cite{A}, locally convex functions are second order differentiable almost everywhere. Thus both $u_1$ and $u_3$ are second order differentiable almost everywhere and $D^2u_1\geq 0, D^2u_3\geq 0$ a.e. 

We first prove that $\det(D^2u_1)=g$ a.e. in $\mathbb{R}^n\setminus \bar{D}$. Let $\bar{x}\in \mathbb{R}^n\setminus \bar{D}$ be a point such that $u_1$ is second order differentiable, without loss of generality, let $\bar{x}=0$. Define 
\begin{align*}
\phi_\delta(x)=u_1(0)+D u_1(0)\cdot x+ \frac{1}{2} x^\prime ( D^2u_1(0)+\delta) x.
\end{align*}
Clearly $\phi_\delta(0)=u_1(0)$. Choose $\delta$ so small such that $\phi_\delta(x)\geq u_1(x)$ near $0$, since $u_1$ is a viscosity subsolution, we have
\begin{align*}
\det( D^2\phi_\delta(0))\geq g(0).
\end{align*}
Sending $\delta$ to $0$, we have $\det(D^2u_1(0) )\geq g(0)$. In particular, $D^2u_1(0)>0$.

Similary, define
\begin{align*}
\phi_\delta(x)=u_1(0)+D u_1(0)\cdot x+\frac{1}{2}  x^\prime ( D^2u_1(0)-\delta) x.
\end{align*}
Clearly $\phi_\delta(0)=u_1(0)$. Choose $\delta$ so small such that $\phi_\delta(x)\leq u_1(x)$ near $0$ and that $\phi_\delta$ is convex, since $u_1$ is a viscosity supsolution, we have
\begin{align*}
\det( D^2\phi_\delta(0))\leq g(0).
\end{align*}
Sending $\delta$ to $0$, we have $\det(D^2u_1)(0) \leq g(0)$.

Thus we have
\begin{align*}
\det(D^2u_1)=g,\quad a.e.
\end{align*}

Similarly
\begin{align*}
\det(D^2u_3)=g,\quad a.e.
\end{align*}

By the concavity of $\det^{\frac{1}{n}}$, we have
\begin{align*}
\det(D^2u_c)^{\frac{1}{n}}\geq \alpha \det(D^2u_3)^{\frac{1}{n}}+(1-\alpha)\det(D^2u_1)^{\frac{1}{n}}=g^{\frac{1}{n}},\quad a.e.
\end{align*}

\medskip

We now prove that $u_c$ is a viscosity subsolution.

Let $\bar{x}\in \Omega$ be an arbitrary point, say $\bar{x}=0$. For any $\phi\in C^2$, such that $\phi(0)=u_c(0)$ and $\phi(x)\geq u_c(x)$ for $x$ near $0$.

For $\delta>0$ small, define $\phi_\delta(x)=\phi(x)+\delta |x|^2$, then 
\begin{align*}
\phi_\delta(x)\geq u_c(x)+\delta^3, \quad for\quad  |x|=\delta.
\end{align*}

Consider
\begin{align*}
\xi(x)=\phi_\delta(x)-u_c(x)-\delta^4.
\end{align*}

Then $\xi(0)=-\delta^4$ and $\xi(x)> 0$ for $|x|=\delta$. Since $D^2u_c\geq 0$ almost everythere, we have $D^2\xi\leq C$ almost everywhere. It follows from the Alexandrov-Bakelman-Pucci inequality that
\begin{align*}
\delta^4\leq C\left(\int_{\{\xi=\Gamma_\xi\}}\det(D^2\Gamma_\xi)\right)^{\frac{1}{n}},
\end{align*}
where $\Gamma_\xi$ is the convex envelope of $\xi$.

Thus 
\begin{align*}
\int_{\{\xi=\Gamma_\xi\}}\det(D^2\Gamma_\xi)>0.
\end{align*}
In particular $\{\xi=\Gamma_\xi\}$ has positive Lebesgue measure. Let $x\in \{\xi=\Gamma_\xi\}\cap B_\delta(0)$ where $u_c$ is second order differentiable and $\det(D^2u_c(x))\geq g(x)$. Thus
\begin{align*}
\xi(y)\geq \Gamma_\xi(x)+\nabla \Gamma_\xi(x)\cdot (y-x),
\end{align*}
and $D^2\xi (x)\geq 0$, i.e.
\begin{align*}
D^2\phi_\delta (x)\geq D^2u_c(x).
\end{align*}

It follows that $\det( D^2\phi_\delta(x) )\geq\det(D^2u_c(x))\geq g(x)$. Sending $\delta$ to $0$, we have $\det( D^2\phi (0))\geq g(0)$. Thus $u_c$ is a viscosity subsolution. 

On the other hand,
\begin{align*}
u_c=\varphi \quad on\quad \partial D,
\end{align*}
and
\begin{align*}
\lim_{x\rightarrow \infty} (u_c- \frac{1}{2}|x|^2)= c.
\end{align*}

Thus $u_c$ is a viscosity subsolution of (\ref{visdirichlet}).

\medskip

Now for every viscosity subsolution $v$ of (\ref{visdirichlet}), we have
\begin{align*}
v= u_3=\varphi \quad on \quad \partial D,
\end{align*}
and 
\begin{align*}
\lim_{|x|\rightarrow \infty}(v(x)-u_3(x))=c-c_3>0.
\end{align*}

We deduce from the comparison principle, applied to $v$ and $u_3+c-c_3$, as before that
\begin{align*}
v-u_3\leq  c-c_3\quad on \quad \mathbb{R}^n\setminus \bar{D}.
\end{align*}

Let 
\begin{align*}
u(x)=\sup\{v(x)|v \text{ is a viscosity subsolution of }  (\ref{visdirichlet})\}.
\end{align*}

Now by the comparison principle, $v\leq u_1$ for all such $v$, thus $u\leq u_1$. On the other hand, $u_c$ is a viscosity subsolution, thus $u\geq u_c$. Then a standard argument shows that $u$ is a viscosity solution of
\begin{align*}
\begin{cases}
\det(D^2u)= g \quad in \quad \mathbb{R}^n\setminus \bar{D},\\
u= \varphi \quad on \quad \partial D.
\end{cases}
\end{align*}

Now we only need to consider the asymptotic behavior. Since for all such $v$, we have
\begin{align*}
\lim_{|x|\rightarrow \infty}(v(x)-\frac{1}{2}|x|^2)=c.
\end{align*}
Thus
\begin{align*}
\liminf_{|x|\rightarrow \infty}(u(x)-\frac{1}{2}|x|^2)\geq  c.
\end{align*}

On the other hand, as $v-u_3\leq  c-c_3$ for all such $v$, we have $u-u_3\leq c-c_3$ in $\mathbb{R}^n\setminus D$, and therefore
\begin{align*}
\limsup_{|x|\rightarrow \infty}(u(x)-\frac{1}{2}|x|^2)\leq \lim_{|x|\rightarrow \infty}(u_3+c-c_3-\frac{1}{2}|x|^2)=c.
\end{align*}

Thus $u$ has the asymptotic behavior $\frac{1}{2}|x|^2+c$. 

\end{proof}

\medskip

Proof of the Theorem \ref{Main thm} and Theorem \ref{thm perturbation}:
\begin{proof}
By Theorem A and Theorem C, there exists a unique viscosity solution for $c> c_1$, without loss of generality, we assume that $c_1$ is the infimum of all such $c_1$. On the other hand,  by Lemma \ref{lemm1} and Lemma \ref{lemm1-1}, there exists $c_2$ such that there is no viscosity solution for $c<c_2$, without loss of generality, we assume that $c_2$ is the supremum of all such $c_2$. Now we prove that $c_1=c_2$. Indeed, if this is not true, then there exists $c_2\leq  c_3< c_1$ such that we have a viscosity solution with asymptotic behavior $\frac{1}{2}x^\prime A x+bx+c_3$, for otherwise it violates the fact that $c_2$ is the supereme of all such $c_2$. Now by Lemma \ref{lemm2}, we know that for all $c\geq c_3$, we have viscosity solution with asymptotic behavior $\frac{1}{2}x^\prime A x+bx+c$. This violates the fact that $c_1$ is the infimum of all such $c_1$.

Now we may denote the unique constant $C_*=c_1=c_2$. By the discussion above, we have viscosity solution for all $c>C_*$ and we have no viscosity solution for all $c<C_*$. We only need to prove that we have viscosity solution for $C_*$.

By the above, we have a sequence of viscosity solutions $u_j$ satisfying
\begin{align*}
\begin{cases}
\det(D^2u_j)=g \quad in \quad \mathbb{R}^n\setminus \bar{D},\\
u_j=\varphi \quad on \quad \partial D,\\
\lim_{|x|\rightarrow\infty}|u_j(x)-\left(\frac{1}{2}x^\prime Ax+bx+C_*+\frac{1}{j}\right)|=0.
\end{cases}
\end{align*}

By comparison principle, we know that $ u_i\geq u_j$, $\forall i\leq j$. Thus $u_j$ is monotonically nonincreasing as $j\rightarrow\infty$.

Define
\begin{align*}
u_\infty(x)=\lim_{j\rightarrow \infty}u_j(x), \quad x\in \mathbb{R}^n\setminus D.
\end{align*}

We claim that $u_\infty$ is a viscosity solution of 
\begin{align}\label{vissup}
\begin{cases}
\det(D^2v)= g \quad in \quad \mathbb{R}^n\setminus \bar{D},\\
v=\varphi \quad on \quad \partial D,\\
\lim_{|x|\rightarrow\infty}|v(x)-\left(\frac{1}{2}x^\prime Ax+bx+C_*\right)|=0.
\end{cases}
\end{align}

We note that by the comparison principle, $u_i\geq u_j-\frac{1}{j}+\frac{1}{i}$ for all $i\geq  j$. Sending $i$ to $\infty$, we have $u_j\geq u_\infty\geq u_j-\frac{1}{j}$ for all $j$. Thus
\begin{align*}
\lim_{j\rightarrow\infty}\|u_j-u_\infty\|_{L^\infty(\mathbb{R}^n\setminus D)}=0.
\end{align*} 
It follows that $\lim_{|x|\rightarrow\infty}|u_\infty(x)-\left(\frac{1}{2}x^\prime Ax+bx+C_*\right)|=0$.

Since $\{u_j\}$ satisfies the first two lines in (\ref{vissup}), and $u_j\rightarrow u_\infty$ uniformly, it is standard that $u_\infty$ also satisfies the first two lines of (\ref{vissup}).

\end{proof}

\section{Further discussions}
In this section, we generalize some results above to more general fully nonlinear elliptic equations. To begin with, let us recall some definitions. We have the following two equivalent definitions for the equations, see \cite{LL} for more detail. 

{\bf Definition 1} : Let $\Gamma \subset \mathbb{R}^n$ be an open convex symmetric cone with vertex at the origin satisfying
\begin{align*}
\Gamma_n\subset \Gamma \subset \Gamma_1:= \{\lambda\in \mathbb{R}^n|\sum_{i=1}^n\lambda_i>0\},
\end{align*}
where $\Gamma_n:=\{\lambda\in \mathbb{R}^n|\lambda_i>0,i=1,\cdots,n\}$ is the positive cone.

Here $\Gamma$ being symmetric means $(\lambda_1,\cdots,\lambda_n)\in \Gamma$ implies $(\lambda_{i_1},\cdots,\lambda_{\lambda_n})\in \Gamma$ for any perturbation $(i_1,\cdots,i_n)$ of $(1,\cdots,n)$.

Assume that $f\in C^1(\Gamma)\cap C^0(\bar{\Gamma})$ is concave and symmetric in $\lambda_i$ satisfying
\begin{align*}
f|_{\partial\Gamma}=0,\quad \nabla f\in \Gamma_n \quad on\quad \Gamma,
\end{align*}
and
\begin{align*}
\lim_{s\rightarrow\infty} f(s\lambda)=\infty,\quad \lambda\in \Gamma.
\end{align*}

Then the equation is
\begin{align*}
f(\lambda(D^2u))=1,\quad \lambda(D^2u)\in \Gamma.
\end{align*}

{\bf Definition 2}: Let $V$ be an open symmetric convex subset of $\mathbb{R}^n$ with $\partial V\neq \emptyset$ and $\partial V\in C^1$. Assume that 
\begin{align*}
\nu(\lambda)\in \Gamma_n,\quad \forall \lambda\in \partial V,
\end{align*}
and
\begin{align*}
\nu(\lambda)\cdot \lambda>0,\quad \forall \lambda\in \partial V,
\end{align*}
where $\nu(\lambda)$ is the inner unit normal of $\partial V$ at $\lambda$.

Then the equation is
\begin{align}\label{visgeneral}
\lambda(D^2u)\in \partial V.
\end{align}

\begin{rema}
In particular, if $f=\det^{\frac{1}{n}}$, then $\Gamma=\Gamma_n$. If $f=\sigma_k^{\frac{1}{k}}$, the $k$-th elementary symmetric function, then $\Gamma=\Gamma_k:=\{\lambda\in \mathbb{R}^n|\sigma_j(\lambda)>0,j=1,\cdots,k\}$. If $f=\left(\frac{\sigma_k}{\sigma_l}\right)^{\frac{1}{k-l}}$ for $k>l$, then $\Gamma=\Gamma_k$.
\end{rema}

\medskip

To proceed, let us recall the definition of viscosity solutions, see \cite{L09, LNW} in this context. We first define the upper semi-continous and lower semi-continuous functions.  Let $S\subset\mathbb{R}^n$, we denote $USC(S)$ the set of functions $\phi: S\rightarrow\mathbb{R}\cup \{-\infty\}$, $\phi\neq -\infty$ in $S$, satisfying
\begin{align*}
\limsup_{x\rightarrow \bar{x}}\phi(x)=\phi(\bar{x}),\quad \bar{x}\in S.
\end{align*}

Similarly, we denote $LSC(S)$ the set of functions $\phi: S\rightarrow\mathbb{R}\cup \{+\infty\}$, $\phi\neq +\infty$ in $S$, satisfying
\begin{align*}
\liminf_{x\rightarrow \bar{x}}\phi(x)=\phi(\bar{x}),\quad \bar{x}\in S.
\end{align*}

\medskip

Let $\Omega$ be a domain in $\mathbb{R}^n$.  For a function $u$ in  $USC(\Omega)$, we say $u$ is a viscosity subsolution in $\Omega$ if for any $x_0\in \Omega$, $\varphi\in C^2(\Omega)$ with 
\begin{align*}
(u-\varphi)(x_0)=0,\quad u-\varphi\leq 0 \quad in \quad \Omega.
\end{align*}
there holds 
\begin{align*}
\lambda(D^2\varphi)\in \bar{V}.
\end{align*}

For a function $u$ in  $LSC(\Omega)$, we say $u$ is a viscosity supsolution in $\Omega$ if for any $x_0\in \Omega$, $\varphi\in C^2(\Omega)$ with 
\begin{align*}
(u-\varphi)(x_0)=0,\quad u-\varphi\geq 0 \quad in \quad \Omega.
\end{align*}
there holds 
\begin{align*}
\lambda(D^2\varphi)\in \mathbb{R}^n\setminus V.
\end{align*}

We say that a function $u\in C^0(\Omega)$ is a viscosity solution of
\begin{align*}
\lambda(D^2u)\in \partial V,
\end{align*}
in $\Omega$ if it is both a viscosity subsolution and a viscosity supsolution.

\medskip

We now state the comparison principle, see Theroem 1.7 in \cite{LNW}.
\begin{lemm}\label{comparison lemma}
Let $\Omega$ be a domain in $\mathbb{R}^n$, assume that $u\in USC(\bar{\Omega})$ and $v\in LSC(\bar{\Omega})$ are respectively a viscosity subsolution and a viscosity supsolution of (\ref{visgeneral}) in $\Omega$ satisfying $u\geq v$ on $\partial\Omega$. Then $u\geq v$ in $\Omega$.
\end{lemm}

\medskip

We now state the first result in this section.

\begin{lemm}\label{lemm3}
Let $D$ be a smooth, bounded, strictly convex domain in $\mathbb{R}^n$ for $n\geq 3$, let $\varphi\in C^2(\partial D)$, let $A$ be a real positive definite $n\times n$ symmetric matrix such that $\lambda(A)\in \partial V$, $b\in \mathbb{R}^n$. Suppose there exist a two constants $c_3<c_1$ such that there exist viscosity solutions $u_3$ and $u_1$ satisfying
\begin{align*}
\begin{cases}
\lambda(D^2u)\in \partial V \quad in \quad \mathbb{R}^n\setminus \bar{D},\\
u=\varphi \quad on \quad \partial D.
\end{cases}
\end{align*}
with asymptotic behavior $\frac{1}{2}x^\prime Ax+bx+c_3$ and $\frac{1}{2}x^\prime Ax+bx+c_1$ respectively as $|x|\rightarrow \infty$. 

Then for any $c_3<c<c_1$, there exists a viscosity solution $u$
satisfying
\begin{align*}
\begin{cases}
\lambda(D^2u)\in \partial V  \quad in \quad \mathbb{R}^n\setminus \bar{D},\\
u=\varphi \quad on \quad \partial D,\\
\lim_{|x|\rightarrow\infty}|u(x)-\left(\frac{1}{2}x^\prime Ax+bx+c\right)|=0.
\end{cases}
\end{align*}
\end{lemm}

\begin{proof}
We note that by an orthogonal transformation and adding a linear function to $u$, we only need to consider case $b=0$ and $A=\Lambda$ is a diagonal matrix. 

Now we show that for any $ c_3< c< c_1$, there is a viscosity solution $u$ with asymptotic behavior $\frac{1}{2}x^\prime \Lambda x+c$.

To show this, we consider viscosity subsolutions $v$ to the exterior Dirichlet problem
\begin{align}\label{visgenerallambda}
\begin{cases}
\lambda(D^2v)\in  \bar{V} \quad in \quad \mathbb{R}^n\setminus \bar{D},\\
v= \varphi \quad on \quad \partial D,\\
\lim_{|x|\rightarrow\infty}|v(x)-\left(\frac{1}{2}x^\prime \Lambda x+c\right)|=0.
\end{cases}
\end{align}

First of all, we show that there exists at least one viscosity subsolution.

Let $u_3$ and $u_1$ be solutions with asymptotic behavior $\frac{1}{2}x^\prime \Lambda x+c_3$ and $\frac{1}{2}x^\prime \Lambda x+c_1$ respectively.

Note that $c=\alpha c_3+(1-\alpha) c_1$ for some $0<\alpha<1$. Let $u_c:=\alpha u_3+(1-\alpha) u_1$, by Lemma \ref{app}
\begin{align*}
\lambda(D^2u_c)\in \bar{V}.
\end{align*}

On the other hand,
\begin{align*}
u_c=\varphi \quad on\quad \partial D,
\end{align*}
and
\begin{align*}
\lim_{x\rightarrow \infty} (u_c- \frac{1}{2}x^\prime \Lambda x)= c.
\end{align*}

Thus $u_c$ is a viscosity subsolution of (\ref{visgenerallambda}).

Now for every viscosity subsolution $v$ of (\ref{visgenerallambda}), we have
\begin{align*}
v= u_3=u_1 \quad on \quad \partial D,
\end{align*}
\begin{align*}
\lim_{|x|\rightarrow \infty}(v(x)-u_1(x))=c-c_1,
\end{align*}
\begin{align*}
\lim_{|x|\rightarrow \infty}(v(x)-u_3(x))=c-c_3.
\end{align*}

We deduce from Lemma \ref{comparison lemma}, applied to $v$ and $u_1$, and then to $v$ and $u_3+c-c_3$, as before that
\begin{align*}
v&\leq  u_1\quad in \quad \mathbb{R}^n\setminus D,\\
v-u_3&\leq  c-c_3\quad in \quad \mathbb{R}^n\setminus D,
\end{align*}
for all such subsolution $v$.

Let 
\begin{align*}
u(x)=\sup\{v(x) | v \text{ is a viscosity subsolution of }  (\ref{visgenerallambda})\}.
\end{align*}

Since $u_c$ is a viscosity subsolution of (\ref{visgenerallambda}), we have 
\begin{align*}
u\geq u_c\quad in\quad \mathbb{R}^n\setminus D.
\end{align*}

By the above, every viscosity subsolution of (\ref{visgenerallambda}) satisfies $v\leq u_1$, we have
\begin{align*}
u\leq u_1\quad in\quad  \mathbb{R}^n\setminus D.
\end{align*}

Since $u_c=u_1=\varphi$ on $\partial D$ and $u_c,u_1\in C^0(\mathbb{R}^n\setminus D)$, we have $u=\varphi$ on $\partial D$.

By the proof of Theorem 1.5 in \cite{LNW},  $u$ is a viscosity solution of 
\begin{align*}
\begin{cases}
\lambda(D^2 u)\in \partial V \quad in \quad \mathbb{R}^n\setminus \bar{D},\\
u= \varphi \quad on \quad \partial D.
\end{cases}
\end{align*}

Now we only need to consider the asymptotic behavior. Since for all such $v$, we have
\begin{align*}
\lim_{|x|\rightarrow \infty}(v-\frac{1}{2}x^\prime \Lambda x)=c.
\end{align*}
It follows that
\begin{align*}
\liminf_{|x|\rightarrow \infty}(u-\frac{1}{2}x^\prime \Lambda x)\geq  c.
\end{align*}

On the other hand, as $v-u_3\leq  c-c_3$ for all such $v$, we have $u-u_3\leq c-c_3$ on $\mathbb{R}^n\setminus D$ and therefore
\begin{align*}
\limsup_{|x|\rightarrow \infty}(u-\frac{1}{2}x^\prime \Lambda x)\leq \lim_{|x|\rightarrow \infty}(u_3+c-c_3-\frac{1}{2}x^\prime \Lambda x)=c
\end{align*}

Thus $u$ has asymptotic behavior $\frac{1}{2}x^\prime \Lambda x+c$. 

\end{proof}

\medskip

We now consider the limiting case.

\begin{lemm}
Let $D$ be a smooth, bounded, strictly convex domain in $\mathbb{R}^n$ for $n\geq 3$, let $\varphi\in C^2(\partial D)$, let $A$ be a real positive definite $n\times n$ symmetric matrix such that $\lambda(A)\in \partial V$, $b\in \mathbb{R}^n$. Suppose there exists a sequence $c_j\rightarrow c$ such that for each $c_j$, there exists a viscosity solution $u_j$ satisfying
\begin{align*}
\begin{cases}
\lambda(D^2u)\in \partial V \quad in \quad \mathbb{R}^n\setminus \bar{D},\\
u=\varphi \quad on \quad \partial D,\\
\lim_{|x|\rightarrow\infty}|u_j(x)-\left(\frac{1}{2}x^\prime Ax+bx+c_j\right)|=0.
\end{cases}
\end{align*}

Then there exists a viscosity solution $u$ satisfying
\begin{align}\label{3.3}
\begin{cases}
\lambda(D^2u)\in \partial V  \quad in \quad \mathbb{R}^n\setminus \bar{D},\\
u=\varphi \quad on \quad \partial D,\\
\lim_{|x|\rightarrow\infty}|u(x)-\left(\frac{1}{2}x^\prime Ax+bx+c\right)|=0.
\end{cases}
\end{align}

\end{lemm}

\begin{proof}
We first note that by passing to a subsequence, we may assume that $c_j$ is monotonically increasing or decreasing. We only need to prove that $c_j$ is monotonically decreasing. The other case follows in the same way.

Applying Lemma \ref{comparison lemma}, we have $u_i-c_i+c_j\leq u_j \leq u_i$ in $\mathbb{R}^n \setminus D$ for $i\leq j$. Thus $\{u_j\}$ is monotonically nonincreasing and $\|u_i-u_j\|_{L^\infty (\mathbb{R}^n \setminus D)}\rightarrow 0$ as $i,j\rightarrow \infty$. Consequently, for some $u_\infty\in C^0(\mathbb{R}^n\setminus D)$, 
\begin{align*}
\lim_{j\rightarrow\infty}\|u_j-u_\infty\|_{L^\infty(\mathbb{R}^n\setminus D)}=0.
\end{align*} 
It follows that
\begin{align*}
\lim_{|x|\rightarrow \infty}|u_\infty(x)-\left(\frac{1}{2}x^\prime Ax+bx+c\right)|=0.
\end{align*}
Since $\{u_j\}$ satisfies the first two lines in (\ref{3.3}), and $u_j\rightarrow u_\infty$ uniformly, it is standard that $u_\infty$ also satisfies the first two lines of (\ref{3.3}).

\end{proof}

\section{Appendix}
In this appendix, we prove the following lemma.

\begin{lemm}\label{app}
Let $u,v$ be two viscosity solutions of (\ref{visgeneral}), then for any $0<\alpha<1$, $\alpha u+(1-\alpha)v$ is a viscosity subsolution of (\ref{visgeneral}).
\end{lemm}

\begin{proof}
To begin with, let us recall the definition of $\epsilon$-upper envelope of $u$, see e.g. \cite{LNW},
\begin{align*}
u^\epsilon(x):=\max_{y\in \bar{\Omega}}\{u(y)-\frac{1}{\epsilon}|y-x|^2\},\quad x\in \bar{\Omega}.
\end{align*}
Then
\begin{align*}
u^\epsilon\rightarrow u, \quad \epsilon\rightarrow 0^+,
\end{align*}
in $C^0_{loc}(\Omega)$.

And $u^\epsilon$ is a viscosity subsolution of (\ref{visgeneral}). 

Moreover $u^\epsilon$ is second order differentiable almost everythere and
\begin{align*}
D^2u^\epsilon\geq -\frac{2}{\epsilon}I,\quad a.e. \quad in\quad \Omega.
\end{align*}

Let $\bar{x}\in \Omega$ be a point where $u^\epsilon$ is second order differentiable, say $\bar{x}=0$. For $\delta>0$ small, define 
\begin{align*}
\phi(x)=u^\epsilon(0)+\nabla u^\epsilon(0)x+\frac{1}{2}x^\prime (D^2 u^\epsilon(0)+\delta) x.
\end{align*}
Then $\phi(x)\geq u^\epsilon(x)$ near $0$ and $\phi(0)=u^\epsilon(0)$. 

Since $u^\epsilon$ is a viscosity subsolution, we have $\lambda(D^2\varphi)(0)\in \bar{V}$. Sending $\delta$ to $ 0$, we have $\lambda(D^2u^\epsilon)(0)\in \bar{V}$. Thus $\lambda(D^2u^\epsilon)\in \bar{V}$ a.e.

Similarly, $\lambda(D^2v^\epsilon)\in \bar{V}$ a.e., where $v^\epsilon$ is the $\epsilon$-upper envelope of $v$.

Define
\begin{align*}
w^\epsilon_\alpha=\alpha u^\epsilon+(1-\alpha)v^\epsilon.
\end{align*}

Since $\bar{V}$ is convex, it follows that $\lambda(D^2 w^\epsilon_\alpha)\in \bar{V}$ a.e.

We now prove that $w^\epsilon_\alpha$ is a viscosity subsolution in $\Omega$.

Let $\bar{x}\in \Omega$ be an arbitrary point, say $\bar{x}=0$. For any $\phi\in C^2$, such that $\phi(0)=w^\epsilon_\alpha(0)$ and $\phi(x)\geq w^\epsilon_\alpha(x)$ for $x$ near $0$.

For $\delta>0$ small, define $\phi_\delta(x)=\phi(x)+\delta |x|^2$, then 
\begin{align*}
\phi_\delta(x)\geq w^\epsilon_\alpha(x)+\delta^3,\quad for\quad |x|=\delta.
\end{align*}

Consider
\begin{align*}
\xi(x)=\phi_\delta(x)-w^\epsilon_\alpha(x)-\delta^4.
\end{align*}

Then $\xi(0)=-\delta^4$ and $\xi(x)> 0$ for $|x|=\delta$. Since $D^2u^\epsilon\geq -\frac{2}{\epsilon}I$, $D^2v^\epsilon\geq -\frac{2}{\epsilon}I$  almost everythere, we have $D^2\xi\leq C(\epsilon)I$ almost everywhere. It follows from the Alexandrov-Bakelman-Pucci inequality that
\begin{align*}
\delta^4\leq C\left(\int_{\{\xi=\Gamma_\xi\}}\det(D^2\Gamma_\xi)\right)^{\frac{1}{n}},
\end{align*}
where $\Gamma_\xi$ is the convex envelope of $\xi$.

As in the previous section, there exists some $x\in \{\xi=\Gamma_\xi\}\cap B_\delta(0)$, where $w^\epsilon_\alpha(x)$ is second order differentiable, $\lambda(D^2w^\epsilon_\alpha)(x)\in \bar{V}$, and $D^2\xi(x)\geq 0$, i.e.
\begin{align*}
D^2\phi_\delta(x)\geq D^2w^\epsilon_\alpha(x).
\end{align*}

It follows that $\lambda( D^2\phi_\delta )(x)\in \bar{V}$. Sending $\delta$ to $0$, we have $\lambda( D^2\varphi )(0)\in \bar{V}$. Thus $w^\epsilon_\alpha$ is a viscosity subsolution. Since $w^\epsilon_\alpha\rightarrow\alpha u+(1-\alpha)v$ in $C^0_{loc}(\Omega)$, sending $\epsilon$ to $0$, it follows that $\alpha u+(1-\alpha)v$ is a viscosity subsolution.

\end{proof}

\end{document}